\documentclass[a4paper,12pt]{article}
\usepackage{amsmath,amsthm,amsfonts,amssymb,graphicx,float,calc,microtype,mathtools,thmtools,anyfontsize,underscore,mathscinet,stmaryrd,wasysym}
\usepackage[usenames,dvipsnames,svgnames,table]{xcolor}
\usepackage{todonotes}
\usepackage[OT2,T1]{fontenc}
\usepackage{lmodern}

\usepackage{hyperref}
\hypersetup{
colorlinks,
linkcolor={black},
citecolor={black},
urlcolor={blue!60!black},
pdftitle={The Product Structure of Squaregraphs},
pdfauthor={Hickingbotham--Jungeblut--Merker--Wood},
colorlinks =true,
filecolor={blue!80!black}
}
\usepackage[lmargin=30mm,rmargin=30mm,tmargin=30mm,bmargin=30mm]{geometry}
\renewcommand{\baselinestretch}{1.12}
\allowdisplaybreaks
\newcommand{\defn}[1]{\textcolor{Maroon}{\protect\emph{#1}}}
\usepackage[shortlabels]{enumitem}
\setlist[itemize]{topsep=0ex,itemsep=0ex,parsep=0ex}
\setlist[enumerate]{topsep=0ex,itemsep=0ex,parsep=0ex}
\usepackage[numbers,sort&compress,longnamesfirst]{natbib}
\def\NAT@spacechar{~}
\makeatother
\usepackage[noabbrev,capitalise]{cleveref}
\crefname{lem}{Lemma}{Lemmas}
\crefname{thm}{Theorem}{Theorems}
\crefname{cor}{Corollary}{Corollaries}
\crefname{prop}{Proposition}{Propositions}
\crefname{conj}{Conjecture}{Conjectures}
\crefname{open}{Open Problem}{Open Problems}
\crefname{obs}{Observation}{Observations}
\crefformat{equation}{(#2#1#3)}
\Crefformat{equation}{Equation #2(#1)#3}
\theoremstyle{plain}
\newtheorem{thm}{Theorem}
\newtheorem{lem}[thm]{Lemma}
\newtheorem{cor}[thm]{Corollary}
\newtheorem{obs}[thm]{Observation}
\newtheorem{prop}[thm]{Proposition}

\theoremstyle{definition}

\renewcommand{\leq}{\leqslant}

\renewcommand{\geq}{\geqslant}

\theoremstyle{definition}

\DeclareMathOperator{\dist}{dist}

\DeclareMathOperator*{\sse}{\subseteq}

\DeclareMathOperator*{\N}{\mathbb{N}}

\DeclareMathOperator*{\Pcal}{\mathcal{P}}

\DeclareMathOperator*{\Lcal}{\mathcal{L}}

\DeclareMathOperator{\tw}{tw}
\DeclareMathOperator*{\pw}{pw}

\newcommand{\Bow}{\mathbin{\textup{\Bowtie}}}
\newcommand{\subsetsim}{\mathrel{\substack{\textstyle\subset\\[-0.2ex]\textstyle\sim}}}
\presetkeys%
{todonotes}%
{inline}{}
\renewcommand{\thefootnote}{\fnsymbol{footnote}}

\begin{document}
\title{\bf The Product Structure of Squaregraphs}
\author{%
	Robert Hickingbotham\footnotemark[3]	\quad
	Paul Jungeblut\footnotemark[4]	\\
	Laura Merker\footnotemark[4] \quad 
	David R. Wood\footnotemark[3]
 }

\maketitle

\begin{abstract}
A \defn{squaregraph} is a plane graph in which each internal face is a $4$-cycle and each internal vertex has degree at least 4. This paper proves that every squaregraph is isomorphic to a subgraph of the semi-strong product of an outerplanar graph and a path. 
We generalise this result for infinite squaregraphs, and show that this is best possible in the sense that ``outerplanar graph'' cannot be replaced by ``forest''.
\end{abstract}

\section{Introduction}
\label{Introduction}

\footnotetext[3]{School of Mathematics, Monash University, Melbourne, Australia (\texttt{\{robert.hickingbotham,david.wood\}@monash.edu}). Research of R.H.\ supported by an Australian Government Research Training Program Scholarship. Research of D.W.\ supported by the Australian Research Council.}

\footnotetext[4]{Institute of Theoretical Informatics, Karlsruhe Institute of Technology, Germany (\texttt{\{paul.jungeblut,laura.merker2\}@kit.edu}).}

\renewcommand{\thefootnote}{\arabic{footnote}}
A \defn{squaregraph} is a plane graph\footnote{A \defn{plane graph} is a graph embedded in the plane with no crossings. The word `face' refers to the subgraph on the boundary of the face. A graph is \defn{outerplanar} if it is isomorphic to a plane graph where every vertex is on the outer-face.} in which each internal face is a $4$-cycle and each internal vertex has degree at least $4$. These graphs were introduced in 1973 by \citet{SZP73}. 
They have many interesting structural and metric properties. For example, \citet{BCE10} showed that squaregraphs are median graphs
and are thus partial cubes, and that every squaregraph can be isometrically embedded\footnote{A graph $H$ can be \defn{isometrically embedded} into a graph $G$ if there exists an isomorphism $\phi$ from $V(H)$ to a subgraph of $G$ such that $\dist_H(u,v)=\dist_G(\phi(u),\phi(v))$ for all $u,v\in V(H)$.} into the cartesian product\footnote{The following are the standard graph products. For graphs $ G $ and $ H $, the \defn{cartesian product} $ G \boxempty H $ is the graph with vertex-set $ V(G) \times V(H) $ with an edge between two vertices $ (v,w) $ and $ (v',w') $ if $ v=v' $ and $ ww' \in E(H) $, or $ w=w' $ and $ vv' \in E(G) $. The \defn{direct product} $ G \times H $ is the graph with vertex-set $ V(G) \times V(H) $ with an edge between two vertices $ (v,w) $ and $ (v',w') $ if $ vv' \in E(G) $ and $ ww' \in E(H) $. The \defn{strong product} $ G \boxtimes H := (G\boxempty H)\cup (G\times H)$. } of five trees. See the survey by \citet{BC08} for background on metric graph theory.


The primary contribution of this paper is the following product structure theorem for squaregraphs, as illustrated in \cref{fig:squaregraph-product}. For graphs $G$ and $H$, the \defn{semi-strong product} \defn{$ G \Bow H $} is the graph with vertex-set $ V(G) \times V(H) $ with an edge between two vertices $ (v,w) $ and $ (v',w') $ if 	$ v=v' $ and $ ww' \in E(H) $, or $ vv' \in E(G) $ and $ ww' \in E(H) $; see \citep{GRW76,HLL21} for example. Note that 
\[G \times H \,\sse\, G \,\Bow\, H \,\sse\, G \boxtimes H.\] 
We write \defn{$H \subsetsim G$} to mean that $H$ is isomorphic to a subgraph of $G$. 



\begin{restatable}{thm}{squaregraphs}
\label{squaregraphs}
For every squaregraph $G$ there is an outerplanar graph $H$ and a path $P$ such that $G\subsetsim H\Bow P$.
\end{restatable}

Note that since a path is bipartite, $H \Bow P$ is also bipartite.

\begin{figure}[h]
    \centering
    \includegraphics{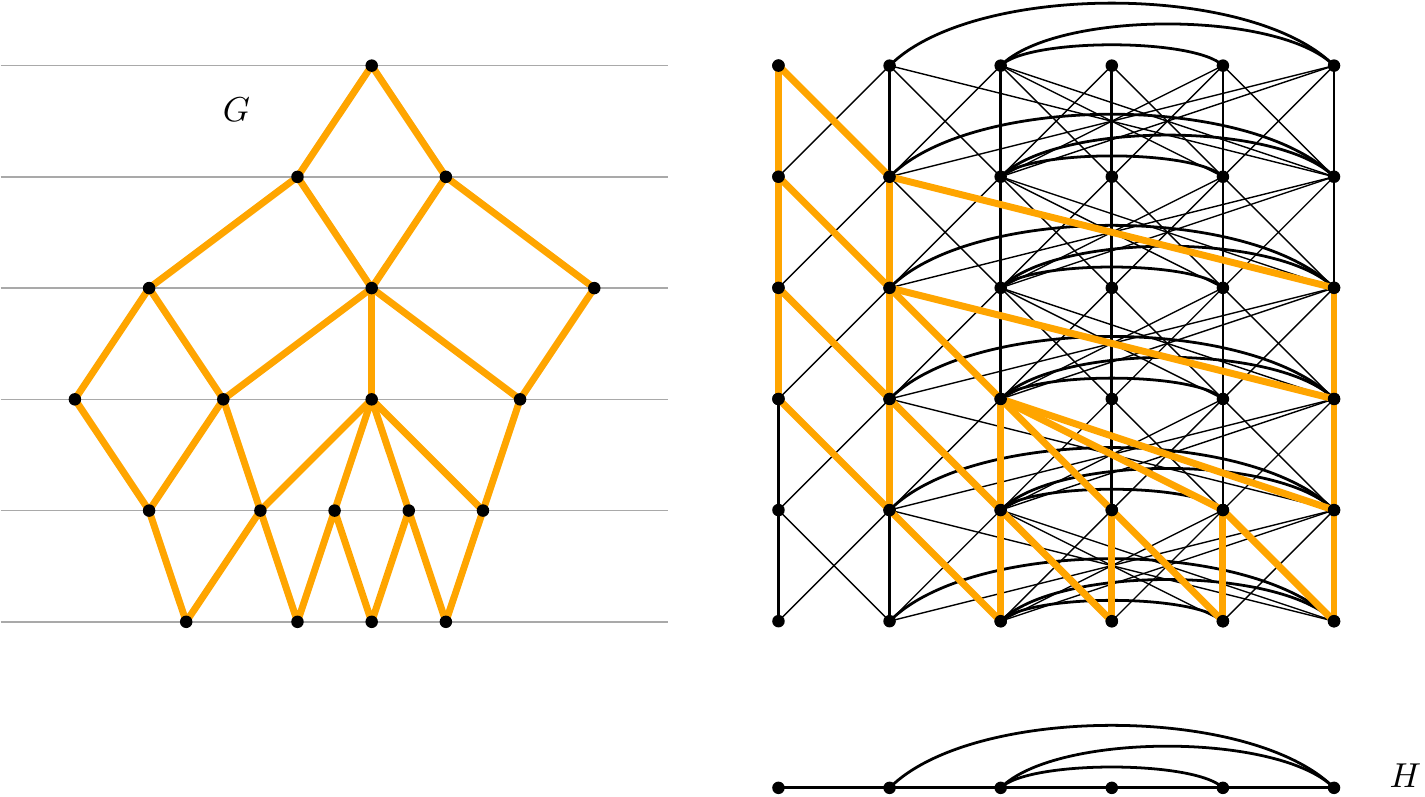}
    \caption{%
        A squaregraph $ G $ (left) isomorphic to a subgraph of the semi-strong product $ H \Bow P $ of an outerplanar graph $ H $ and a path $P$ (right).
    }
    \label{fig:squaregraph-product}
\end{figure}

We in fact prove a more general sufficient condition for a plane graph to have such a product structure which implies \cref{squaregraphs}; see \cref{srtw2-bfs} in \cref{SectionUB}. 

The second contribution of this paper is to show that \cref{squaregraphs} is best possible in the sense that ``outerplanar graph'' cannot be replaced by ``forest''. Moreover, this lower bound holds for strong products. In fact, we prove that for every integer $\ell\in\mathbb{N}$ there is a squaregraph $G$ such that for any graph $H$ and path $P$, if $G\subsetsim H\boxtimes P\boxtimes K_\ell$ then $H$ contains a cycle (and is therefore not a forest). This result actually follows from a stronger lower bound for bipartite graphs, which has other interesting consequences; see \cref{BipartiteLower} in \cref{SectionLB}. Also note that \cref{squaregraphs} cannot be strengthened by replacing ``outerplanar graph'' by ``graph with bounded pathwidth''. Indeed, \citet{BDJMW} showed that for every $k \in\mathbb{N}$ there is a tree $T$ (which is a squaregraph) such that for any graph $H$ and path $P$, if $T\subsetsim H \boxtimes P$ then $\pw(H)\geq k$. 

In \cref{squaregraphs} it is natural to ask whether there is such an outerplanar graph $H$ independent of $G$. This leads to the study of infinite squaregraphs, previously investigated by \citet{BCE10}. Our final contribution is an extension of \cref{squaregraphs} in which we show that every (possibly infinite) squaregraph is isomorphic to a subgraph of $O\Bow \overrightarrow{P}$, where $O$ is the universal outerplanar graph and $\overrightarrow{P}$ is the 1-way infinite path; see \cref{Infinite}.


Before proving the above results, we provide further motivation by putting  \cref{squaregraphs} in context. The study of the product structure of graph classes emerged with the following seminal result by \citet{DJMMUW20}, now called the \emph{Planar Graph Product Structure Theorem}. This result describes planar graphs in terms of the strong product of graphs with bounded treewidth\footnote{A \defn{tree-decomposition} of a graph $G$ is a collection $(B_x\subseteq V(G):x\in V(T))$ of subsets of $V(G)$ (called \defn{bags}) indexed by the nodes of a tree $T$, such that (a) for every edge $uv\in E(G)$, some bag $B_x$ contains both $u$ and $v$, and (b) for every vertex $v\in V(G)$, the set $\{x\in V(T):v\in B_x\}$ induces a non-empty subtree of $T$. The \defn{width} of a tree-decomposition is the size of the largest bag minus~$1$. The \defn{treewidth} of a graph $G$, denoted by \defn{$\tw(G)$}, is the minimum width of a tree-decomposition of $G$. A \defn{path-decomposition} of a graph $G$ is a tree decomposition $(B_x\subseteq V(G):x\in V(T))$ where $T$ is a path. The \defn{pathwidth} of a graph $G$, denoted by \defn{$\pw(G)$}, is the minimum width of a path-decomposition of $G$.} and a path. A connected graph has treewidth at most 1 if and only if it is a tree. Treewidth measures how similar a graph is to a tree and is an important parameter in algorithmic and structural graph theory; see \cite{HW17,Reed97}. Graphs with bounded treewidth are considered to be a relatively simple class of graphs. 

\begin{thm}[\cite{UWY,DJMMUW20}]\label{PGPST}
For every planar graph $G$ there is a graph $H$ of treewidth at most $6$ and a  path $P$ such that $G\subsetsim H \boxtimes P$.
\end{thm}

The original version of the Planar Graph Product Structure Theorem by \citet{DJMMUW20} had ``treewidth at most $8$'' instead of ``treewidth at most 6''. \citet{UWY} proved \cref{PGPST} with ``treewidth at most $6$''. Since outerplanar graphs have treewidth at most $2$, \Cref{squaregraphs} is stronger than \cref{PGPST} in the case of squaregraphs. \Cref{squaregraphs} is also stronger than \cref{PGPST} in the sense that \Cref{squaregraphs} uses $\Bow$ whereas \cref{PGPST} uses $\boxtimes$. That said, as explained in \cref{Preliminaries}, it is well-known that in the case of bipartite planar graphs $G$, the proof of \cref{PGPST} can be adapted to show that $G\subsetsim H\Bow P$.

Product structure theorems are useful since they reduce problems on a complicated class of graphs (such as planar graphs or squaregraphs) to a simpler class of graphs (bounded treewidth graphs such as outerplanar graphs). They have been the key tool to resolve several open problems regarding queue layouts~\citep{DJMMUW20}, 
nonrepetitive colourings~\citep{DEJWW20}, 
centered colourings~\citep{DFMS21}, 
clustered colourings~\citep{DEMWW22}, 
adjacency labellings~\citep{BGP20,DEJGMM21,EJM}, 
vertex rankings~\citep{BDJM}, 
twin-width~\citep{BKW}, 
odd colourings~\citep{DMO}, 
and infinite graphs \cite{HMSTW}. 
Similar product structure theorems are known for other classes including graphs with bounded Euler genus~\cite{DJMMUW20,DHHW}, 
apex-minor-free graphs~\cite{DJMMUW20}, 
$(g,d)$-map graphs~\cite{DMW}, 
$(g,\delta)$-string graphs~\cite{DMW}, 
$(g,k)$-planar graphs~\cite{DMW}, 
powers of planar graphs~\cite{DMW,HW21b}, 
$k$-semi-fan-planar graphs~\cite{HW21b} 
and $k$-fan-bundle planar graphs~\cite{HW21b}. 

\subsection{Preliminaries}
\label{Preliminaries}

We consider undirected simple graphs $G$ with vertex-set $V(G)$ and edge-set $E(G)$. Unless stated otherwise, graphs are finite. Undefined terms and notation can be found in Diestel's textbook~\citep{Diestel5}. 

For $m,n \in \mathbb{Z}$ with $m \leq n$, let $[m,n]:=\{m,m+1,\dots,n\}$ and $[n]:=[1,n]$. 

Let $P_n$ denote a path on $n$ vertices. For graphs $G$ and $H$, the \defn{complete join $G+H$} is the graph obtained by the disjoint union of~$G$ and~$H$ by adding all edges between~$G$ and~$H$. For a graph $G$ with $A,B\sse V(G)$, let \defn{$G[A,B]$} be the subgraph of $G$ with $V(G[A,B]):=A \cup B$ and $E(G[A,B]):=\{uv \in E(G):u \in A, v\in B\}$. 

A \defn{matching} $M$ in a graph $G$ is a set of edges in $G$ such that no two edges in $M$ have a common endvertex. A matching $M$ \defn{saturates} a set $S\sse V(G)$ if every vertex in $S$ is incident to some edge in $M$.

A \defn{model} of $H$ in $G$ is a function $\mu$ with domain $V(H)$ such that: $\mu(v)$ is a connected subgraph of $G$; $\mu(v)\cap \mu(w)=\emptyset$ for all distinct $v,w\in V(H)$; and $\mu(v)$ and $\mu(w)$ are adjacent for every edge $vw \in E(H)$. If, for some $s \in \mathbb{N}_0$, there is a model $\mu$ of $H$ in $G$ such that $|V(\mu(v))|\leq s$ for each $v \in V(H)$, then $H$ is an \defn{$s$-small minor} of $G$.

In a plane graph $G$, a vertex is \defn{outer} if it is on the outer-face of $G$ and is \defn{inner} otherwise. Let \defn{$I_G$} denote the set of inner vertices in $G$. 

Let $G$ be a graph. A \defn{partition} of $G$ is a set $\Pcal$ of sets of vertices in $G$ such that each vertex of $G$ is in exactly one element of $\Pcal$. Each element of $\Pcal$ is called a \defn{part}. The \defn{quotient} of $\Pcal$ (with respect to $G$) is the graph, denoted by \defn{$G/\Pcal$}, with vertex set $\Pcal$ where distinct parts $A,B\in \mathcal{P}$ are adjacent in $G/\Pcal$ if and only if some vertex in $A$ is adjacent in $G$ to some vertex in $B$. An \defn{$H$-partition} of $G$ is a partition $\Pcal=(A_x:x \in V(H))$ where $H\cong G/\Pcal$. 
For an $H$-partition $(A_x:x\in V(H))$ of $G$, for each subgraph $J\sse G$ the quotient $\tilde{H}$ of the partition $(A_x\cap V(J):x\in V(H),A_x\cap V(J)\neq\emptyset)$ is called the \defn{sub-quotient} for $J$. Note that $\tilde{H}$ is a subgraph of $H$. 

A \defn{layering} of a graph $G$ is an ordered partition $\mathcal{L}:=(L_0,L_1,\dots)$ of $V(G)$ such that for every edge $vw \in E(G)$, if $v \in L_i$ and $w \in L_j$, then $|i-j|\leq 1$. $\mathcal{L}$ is a \defn{\textsc{bfs}-layering} (of $G$) if $L_0 = \{r\}$ for some \defn{root vertex} $r\in V(G)$ and $L_i=\{v\in V(G):\dist_G(v,r)=i\}$ for all $i\geq 1$. A path $P$ is \defn{vertical} (with respect to $\mathcal{L}$) if $|V(P)\cap L_i|\leq 1$ for all $i\geq 0$. 

A \defn{layered partition} $(\Pcal,\mathcal{L})$ of a graph $G$ consists of a partition $\Pcal$ and a layering~$\mathcal{L}$ of $G$. If $\Pcal$ is an $H$-partition, then $(\Pcal,\mathcal{L})$ is a \defn{layered $H$-partition}. If 
$\Pcal=(A_x:x\in V(H))$, then the \defn{width} of $(\Pcal,\mathcal{L})$ is $\max\{|A_x\cap L|:x\in V(H), L \in \mathcal{L}\}$. Layered partitions of width at most $1$ are \defn{thin}. Layered partitions were introduced by \citet{DJMMUW20} who observed the following connection to strong products (which follows directly from the definitions). 

\begin{obs}[\cite{DJMMUW20}]\label{OrthogonalPartitions}
    For all graphs $G$ and $H$, 
    $G \subsetsim H\boxtimes P\boxtimes K_{\ell}$ for some path $P$ if and only if $G$ has a layered $H$-partition $(\Pcal,\mathcal{L})$ with width at most $\ell$. 
\end{obs}


We have the following analogous observation for $\Bow$ (which also follows directly from the definitions). 

\begin{obs}\label{BowPartitions}
    For all graphs $G$ and $H$, 
    $G \subsetsim (H\boxtimes K_\ell) \Bow P$ for some path $P$ if and only if $G$ has a layered $H$-partition $(\Pcal,\mathcal{L})$ with width at most $\ell$, such that each $L \in \mathcal{L}$ is an independent set in $G$. 
\end{obs}


In \cref{BowPartitions} we may use $G \subsetsim (H \boxtimes K_{\ell}) \Bow P $ instead of $G \subsetsim H\boxtimes K_{\ell} \boxtimes P$ when each $L\in\Lcal$ is an independent set, since no edges in $G$ correspond to edges in $H\boxtimes K_{\ell} \boxtimes P$ of the form $ (v,x,w)(v',y,w) $ where $vv'\in E(H)$, $x,y\in V(K_{\ell})$ and $w \in V(P)$.


As mentioned in \cref{Introduction}, it is well-known that in the case of bipartite planar graphs $G$, the proof of \cref{PGPST} can be adapted to show that $G\subsetsim H\Bow P$ for some graph $H$ of treewidth at most $6$ and for some path $P$. To see this, we may assume that $G$ is edge-maximal bipartite planar. Thus $G$ is connected, and each face is a 4-cycle. Let $\Lcal=(L_0,L_1,\dots)$ be a \textsc{bfs}-layering of $G$. So each $L_i$ is an independent set. Each face can be written as $(a,b,c,d)$ where $a\in L_i$ and $b,d\in L_{i+1}$ and $c\in L_i \cup L_{i+2}$, for some $i\geq 0$. Let $G'$ be the planar triangulation obtained from $G$ by adding the edge $bc$ across each such face. Thus $(L_0,L_1,\dots)$ is a layering of $G'$. The proof of \cref{PGPST} shows that $G'$ has a partition $\Pcal$ such that $\tw(G/\Pcal)\leq 6$ and $(\Pcal,\Lcal)$ is a thin layered partition. By construction, $(\Pcal,\Lcal)$ is a layered partition of $G$. By \cref{BowPartitions}, $G \subsetsim H \Bow P$.


A \defn{red-blue colouring} of a bipartite graph $G$ is a proper vertex $2$-colouring of $G$ with colours `red' and `blue'. 

\section{Sufficient Conditions}
\label{SectionUB}

In this section we prove \cref{squaregraphs}. We first prove the following, more general sufficient condition for a plane graph to be isomorphic to a subgraph of the strong or semi-strong product of an outerplanar graph and a path. Afterwards, we show that this more general result implies \cref{squaregraphs}.

\begin{thm}\label{srtw2-bfs}
    Let~$G$ be a plane graph with inner vertices~$I_G$. If $ G $ has a layering $\mathcal{L}= (L_0,L_1,\dots)$ such that $G[L_{i - 1},L_i]$ has a matching saturating $L_{i - 1}\cap I_G$ for each $i \in [n]$, then $G \subsetsim H \boxtimes P$ for some outerplanar graph $H$ and path $P$. Moreover, if $V(L_i)$ is an independent set for all $L_i \in \mathcal{L}$, then $G \subsetsim H \Bow P$.
\end{thm}

\begin{proof}
    By \cref{OrthogonalPartitions,BowPartitions}, it suffices to show that $G$ has a thin layered $H$-partition $\Pcal$ (with respect to $\mathcal{L}$) for some outerplanar graph $H$. 
    For each $i \in [n]$, let $E_{i}$ be a matching in $G[L_{i-1},L_i]$ that saturates $ L_{i - 1}\cap I_G$. For vertices $u \in L_{i-1}$ and $v \in L_i$ and an edge $uv\in E_{i}$, we say that $u$ is the \defn{parent} of $v$ and $v$ is the \defn{child} of $u$. 
    Observe that each vertex $u \in L_{i-1}\cap I_G$ has exactly one child and each vertex $v \in L_i$ has at most one parent.
	Let $J$ be the subgraph of $G$ where $V(J)=V(G)$ and $E(J)=\bigcup_{i \in [n]} E_i$. 
	
	Let $X$ be a connected component of $J$.
	Choose the maximum $j \in [0,n]$ such that there exists some vertex $v \in V(X)\cap L_j$.
	Vertex~$v$ must be outer because each vertex in $L_j \cap I_G$ is adjacent in~$J$ to some vertex in~$L_{j+1}$. As illustrated in \cref{fig:path-contraction}, since each vertex in $X$ has at most one child and at most one parent, $X$ is a vertical path with respect to $\mathcal{L}$.
	
	\begin{figure}[!ht]
		\centering%
    \includegraphics[page = 3]{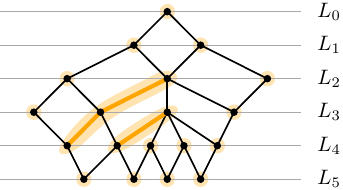} \qquad
    \includegraphics[page = 6]{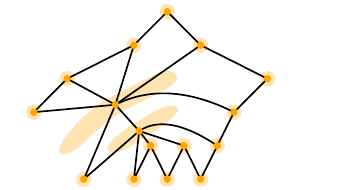}
    	\caption{%
            Left: A squaregraph with a \textsc{bfs}-layering and a partition $ \Pcal $ into vertical paths (thick orange). The vertical paths are constructed from matchings between consecutive layers, where the leftmost vertex in $ L_i $ is chosen for each inner vertex in $ L_{i-1} $. 
            Right: The lower endpoint of each path is on the outer-face, so when each path is contracted we obtain an outerplanar graph. 
        }
		\label{fig:path-contraction}
	\end{figure}
	
	Let $\Pcal$ be the partition of $G$ determined by the connected components of $J$. Let $H=G/\Pcal$ be the quotient of $\Pcal$. Since each part in $\Pcal$ is a vertical path with respect to $\mathcal{L}$, it follows that $(\Pcal,\mathcal{L})$ is a thin layered $H$-partition. It remains to show that $H$ is outerplanar. Since each part in $\Pcal$ is connected, $H$ is a minor of $G$ and is therefore planar. Since each part of $\Pcal$ contains a vertex on the outer-face, contracting each part of $\Pcal$ into a single vertex gives a plane embedding of $H$ with each vertex on the outer-face; see \cref{fig:path-contraction}. Therefore $H$ is outerplanar.
	%
\end{proof}

We now work towards showing that squaregraphs satisfy the conditions for \cref{srtw2-bfs}. 

A plane graph $G$ is \defn{leveled} if the edges are straight line-segments and vertices are placed on a sequence of horizontal lines, $(L_0,L_1,\dots)$, called \defn{levels}, such that each edge joins two vertices in consecutive levels. If, in addition, we allow straight-line edges between consecutive vertices on the same level, then $G$ is \defn{weakly leveled}. Observe that the levels in a weakly leveled plane graph $G$ define a layering of $G$. Leveled plane graphs were first introduced by \citet{STT81}, and have since been well studied \cite{BDDEW19}.




For a weakly leveled plane graph $G$ with levels $(L_0,L_1,\dots)$ and a vertex $v\in L_i$, the \defn{up-degree} of $v$ is $|N_G(v)\cap L_{i-1}|$ and the \defn{down-degree} of $v$ is $|N_G(v)\cap L_{i+1}|$. We now give a more natural condition that forces our desired matching between two consecutive levels. 

\begin{lem}\label{WeaklyLeveledUp}
Let~$G$ be a weakly leveled plane graph with inner vertices $I_G$. If each vertex in $I_G$ has down-degree at least $2$, then $G \subsetsim H \boxtimes P$ for some outerplanar graph $H$ and path $P$. Moreover, if $G$ is a leveled plane graph, then $G \subsetsim H \Bow P$.
\end{lem}

\begin{proof}
    Let $(L_0,L_1,\dots)$ be the levels of $G$. Observe that if $G$ is a leveled plane graph, then $V(L_i)$ is an independent set for all $i\geq 0$. For each $ i \in [n]$, let $E_i$ be the set of edges in $G[L_{i - 1},L_i]$ between each vertex $v \in L_{i-1}\cap I_G$ and its leftmost neighbour in $ L_i $; see \cref{fig:path-contraction}. For the sake of contradiction, suppose there exists a vertex $u\in L_{i - 1}\cup L_i$ that is incident to two edges in $E_i$. By construction, each vertex in $L_{i-1}\cap I_G$ is incident to at most one edge in $E_i$ so $u\in L_i$. Let $x$ and $y$ be the neighbours of $u$ in $L_{i-1}$, where $x$ is to the left of $y$. Since $x$ has down-degree at least $2$, $x$ is adjacent to a vertex $v$ that is to the right of $u$. However, this contradicts $G$ being weakly leveled plane since $uy$ and $vx$ cross; see \cref{fig:LevelCrossing}. Therefore, $E_i$ is a matching that saturates $L_{i-1}\cap I_G$. The claim therefore follows by \cref{srtw2-bfs}.
\end{proof}

\begin{figure}[h!]
    \centering 
    \includegraphics[width=0.21\textwidth]{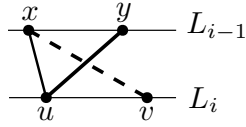}
    \caption{Contradiction in the proof of \Cref{WeaklyLeveledUp}.}
    \label{fig:LevelCrossing}
\end{figure}

We are ready to prove \cref{squaregraphs} which we restate here for convenience.

\squaregraphs*


\begin{proof}
    We may assume that $G$ is connected (since if each component of $G$ has the desired product structure, then so does $G$). By taking a \textsc{bfs}-layering of $G$ rooted at any vertex $r$ on the outer-face, \citet{BDDEW19} showed that $G$ is isomorphic to a leveled plane graph. Without loss of generality, assume $G$ is leveled plane with corresponding levels $(L_0,L_1,\dots)$. 
    
    
    Below we show that every inner vertex in $G$ has up-degree at most 2. Since each inner vertex has degree at least $4$, each inner vertex has down-degree at least $2$. The result thus follows from \cref{WeaklyLeveledUp}.
    
    For the sake of contradiction, suppose there exists an inner vertex with up-degree at least $3$. Let $ i \in [n] $ be minimum such that there is a vertex $ v \in L_i\cap I_G $ with up-degree at least $3$. Let $ u_1, u_2, u_3 $ be neighbours of $v$ in $L_{i-1}$ ordered left to right. Since the levels are defined by a \textsc{bfs}-layering, there is an $(u_1,r)$-path and an $(u_3,r)$-path that does not contain $ u_2 $; see \cref{fig:three-parents}. Hence, $ u_2 $ is an inner vertex of $G$ and thus has degree at least $4$. However, by planarity, $ v $ is the only neighbour of $ u_2 $ in $ L_i $. Since $ u_2 $ has no neighbours in $L_{i-1}$ (as $G$ is leveled plane), $u_2$ has three neighbours in $ L_{i - 2} $, which contradicts the minimality of $ i $, as required. 
\end{proof}

\begin{figure}[!ht]
\centering
\includegraphics{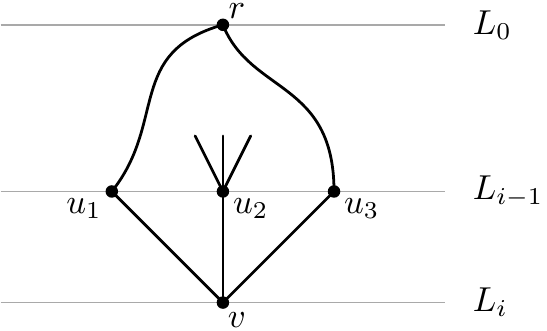}
\caption{Vertex $ v \in L_i $ with three neighbours $ u_1, u_2, u_3 $ in the preceding layer $ L_{i - 1} $. Since $ u_2 $ is an inner vertex, it has degree at least 4.}
\label{fig:three-parents}
\end{figure} 

We now give an application of \cref{squaregraphs}. A colouring $\phi$ of a graph $G$ is \defn{nonrepetitive} if for every path $v_1,\dots,v_{2h}$ in $G$, there exists $i \in [h]$ such that $\phi(v_i)\neq \phi(v_{i+h})$. The \defn{nonrepetitive chromatic number}, $\pi(G)$, is the minimum number of colours in a nonrepetitive colouring of $G$. Nonrepetitive colourings were introduced by \citet{AGHR02} and have since been widely studied; see the survey \citep{Wood21}.

\citet{KP-DM08} showed that $\pi(G)\leq 4^{\tw(G)}$ for every graph $G$. Building upon this result, \citet{DEJWW20} proved the following:
	
 \begin{lem}[\citep{DEJWW20}]\label{NonrepProduct}
 For any graph $H$ and path $P$, if $G \subsetsim H \boxtimes P$ then $\pi(G)\leq 4^{\tw(H)+1}$.
 \end{lem}
 

Using (a variation of) \cref{PGPST,NonrepProduct}, \citet{DEJWW20} resolved a long-standing conjecture of \citet{AGHR02} by showing that planar graphs $G$ have bounded nonrepetitive chromatic number; in particular, $\pi(G) \leq 768$. When $G$ is a squaregraph, \Cref{squaregraphs,NonrepProduct} imply that $\pi(G)\leq 4^3=64$.

\section{Tightness}\label{SectionLB}

In this section, we show that \cref{squaregraphs} is tight by proving a lower bound for the product structure of bipartite graphs.

The \defn{row treewidth} of a graph $G$ is the minimum integer $k$ such that $G\subsetsim H \boxtimes P$ for some graph $H$ with treewidth $k$ and path $P$ \cite{BDJMW}. \cref{PGPST} says that every planar graph has row treewidth at most $6$. \citet{DJMMUW20} showed that the maximum row treewidth of planar graphs is at least $3$. They in fact proved the following stronger result.

\begin{thm}[\cite{DJMMUW20}]
    \label{rtwLB}
    For all $k,\ell\in\mathbb{N}$ with $k\geq 2$ there is a graph $G$ with pathwidth $k$ such that for any graph $H$ and path $P$, if $G \subsetsim H \boxtimes P \boxtimes K_{\ell}$  then $K_{k+1}\subsetsim H$ and thus $H$ has treewidth at least $k$. Moreover, if $k=2$ then $G$ is outerplanar, and if $k=3$ then $G$ is planar.
\end{thm}

\cref{squaregraphs} says that squaregraphs have row treewidth at most 2.
We show that this bound is tight by proving \cref{BipartiteLower} which is an analogous result to \cref{rtwLB} for bipartite graphs. As an introduction to the key ideas in the proof of \cref{BipartiteLower}, we first establish \cref{LowerBoundSubgraph} which is a slight generalisation of \cref{rtwLB}. We need the following lemma for finding long paths in quotient graphs.

\begin{lem}\label{lem:BaseCaseLB}
Let $ G $ be a graph and $(A_x : x\in V(H))$ be an $H$-partition of $G$ such that $|A_x|\leq a$ for all $x \in V(H)$. Then for every $w \in V(H)$ and $ n \in \N $, there is a sufficiently large $ n' \in \N $ such that if $ G $ contains a path on $ n' $ vertices, then $ H-w$ contains a path on $ n $ vertices.
\end{lem}

\begin{proof}
Let $m$ be sufficiently large compared to $n$ and let $n':=(a+1)am+a$. 
Suppose $G$ has a path on $n'$ vertices. Let $G'=G-A_w$.
Since $|V(P)\cap A_w|\leq a$, $P$ is split into at most $a + 1$ disjoint subpaths in $G'$. 
Thus, there is a path $P_{\max}$ in $G'$ with at least $am$ vertices. 
Let $\tilde{H}$ be the sub-quotient of $H$ with respect to $P_{\max}$.
Observe that $\tilde{H}$ is connected and that $|V(\tilde{H})|\geq am/a = m $.
Moreover, $\tilde{H}\sse H-w$ since $A_w\cap V(P_{\max})=\emptyset$.
Now $\tilde{H}$ has maximum degree at most $2a$ since every vertex in $P_{\max}$ has degree at most~$2$. 
Thus, since $m$ is sufficiently large, $\tilde{H}$ contains a path on at least $n$ vertices, as required.
\end{proof}


The following result generalises \cref{rtwLB} (which is the $n= 2$ case). 


\begin{prop}\label{LowerBoundSubgraph}
    For all $k, \ell,n \in \N$ there exists a graph $G$ with pathwidth at most $k+1$ such that for any graph $H$ and path $P$, if $G \subsetsim H \boxtimes P \boxtimes K_{\ell}$ then $P_n+K_k\subsetsim H$.
\end{prop}

\begin{proof}
    We proceed by induction on $k\geq 1$.
    Let $ n' $ be sufficiently large compared to $n$. Let~$G^{(1)}$ be the graph obtained from a path on $n'$ vertices plus a dominant vertex $v$. 
	Observe that $ G^{(1)} $ has radius 1 and pathwidth at most $2$.
	Suppose $ G^{(1)} \subsetsim H \boxtimes P \boxtimes K_{\ell}$ for some graph~$H$ and path~$P$. 
	By \cref{OrthogonalPartitions}, there is a layered $H$-partition $(A_x : x \in V(H))$ of~$G$ of width~$\ell$. Let $w\in V(H)$ be such that $v \in A_w$.
	Since $G^{(1)}$ has radius $1$, every layering of $G^{(1)}$ consists of at most three layers so $|A_x|\leq 3\ell$ for all $x \in V(H)$.
	By \cref{lem:BaseCaseLB} and since $n'$ is sufficiently large, $ H-w $ contains a path on $ n $ vertices.
	As~$v$ is dominant in $ G^{(1)} $, $w$ is also dominant in $ H $.
	Thus $P_n+K_1 \subsetsim H$.

	Now suppose $k > 1$ and let $G^{(k-1)}$ be a graph that satisfies the induction hypothesis for $k-1$.
	Let~$G^{(k)}$ be obtained by taking $3\ell$ disjoint copies of~$G^{(k-1)}$ plus a dominant vertex~$v$. 
	Then $G^{(k)}$ has pathwidth at most $k+1$.
	As in the base case, let $(A_x : x \in V(H))$ be a layered~$H$-partition of~$G^{(k)}$ of width~$\ell$.
	Let $w \in V(H)$ be such that $v \in A_w$.
	Since $G^{(k)}$ has radius $1$, it follows that $|A_x-\{v\}|\leq 3\ell-1$ for all $x \in V(H)$.
	Thus, there is a copy of $G^{(k-1)}$ that contains no vertices from~$A_w$.
	Now consider the sub-quotient $\tilde{H}$ of $H$ with respect to this copy of $G^{(k-1)}$.
	By induction, $P_n+K_{k-1}\subsetsim \tilde{H}$.
	Since~$v$ is dominant in $G^{(k)}$, $w$ is dominant in $H$ and thus $P_n+K_k\subsetsim H$, as required.
\end{proof}

Note that in \cref{LowerBoundSubgraph}, the graph $G^{(1)}$ is outerplanar and the graph $G^{(2)}$ is planar for every $n \in \N$.

We now prove our main lower bound which is a bipartite version of \cref{LowerBoundSubgraph}. 

\begin{thm}
    \label{BipartiteLower}
    For all $i,j,k, \ell,n \in \N$ where $i+j=k$, there exists a bipartite graph $G$ with pathwidth at most $k+1$ such that 
    for any graph $H$ and path $P$, 
    if $G \subsetsim H \boxtimes P \boxtimes K_{\ell}$ then $P_n+K_{i,j}$ is a $2$-small minor of $H$. 
\end{thm}
\begin{proof}
    Let $P_n=(a_1,\dots,a_n)$ be a path on $n$ vertices.
    Let $B=\{b_1,\dots,b_i\}$ and $C=\{c_1,\dots,c_j\}$ be the bipartition of $V(K_{i,j})$. We proceed by induction on $k$ with the following hypothesis: for every $i,j,k, \ell,n \in \N$ where $i+j=k$, there exists a red-blue coloured connected bipartite graph $G$, such that for any graph $H$, if $(A_x : x \in V(H))$ is a layered $H$-partition of $G$  of width $\ell$, then $H$ contains a model $\mu$ of $P_n+K_{i,j}$ such that for each $u \in V(P_n+K_{i,j})$ we have $|V(\mu(u))|\leq 2$ and $\bigcup_{a \in V(\mu(u))}A_a$ contains: 
    \begin{enumerate}
        \item a red vertex when $u\in B$;  
        \item a blue vertex when $u \in C$; and
        \item a red and a blue vertex when $u \in V(P_n)$.
    \end{enumerate}
    The claimed theorem follows by \cref{OrthogonalPartitions}. 
        
	For~$k = 1$ we may assume that $i=1$ and $j=0$.
    Let $ n' $ be sufficiently large and let~$G^{(1,0)}$ be the bipartite graph obtained from a red-blue coloured path $P_G=(u_1,\dots,u_{n'})$ on $n'$ vertices plus a red vertex $v$ adjacent to all the blue vertices in $V(P_G)$.
	Observe that $ G^{(1,0)} $ has radius 2 and pathwidth at most $2$.
	Let $(A_x : x \in V(H))$ be a layered $H$-partition of~$G^{(1,0)}$ of width~$\ell$.
	Let $w\in V(H)$ be such that $v \in A_w$.
	Then $A_w$ contains a red vertex.
	Since $G^{(1,0)}$ has radius $2$, every layering of $G^{(1,0)}$ has at most five layers, so $|A_x|\leq 5\ell$ for all $x \in V(H)$.
	By \cref{lem:BaseCaseLB} and since $n'$ is sufficiently large, $ H-w $ contains a path $P_H=(a_1',\dots,a_{2n}')$ on $ 2n $ vertices.
	Now for every edge $a'_i a'_{i+1}\in E(P_H)$, there exists $j \in [n'-1]$ such that $u_j,u_{j+1}\in A_{a_i'}\cup A_{a_{i+1}'}$. As such, $A_{a_i'}\cup A_{a_{i+1}'}$ contains a red and a blue vertex. For all $i \in [n]$, let $\mu(a_i)=H[\{a_{2i-1}',a_{2i}'\}]$ and $\mu(b_1)=\{w\}$. Then $\mu$ is a model of $P_n+K_{1,0}$ in $H$ which satisfies the induction hypothesis.
	
	Now suppose $k > 1$ and that there is a red-blue coloured connected bipartite graph $G^{(i-1,j)}$ such that for any graph $H$, if $(A_x : x \in V(H))$ is a layered $H$-partition of $G$ of width $\ell$, then $H$ contains a model $\tilde{\mu}$ of $P_n+K_{i-1,j}$ where $|V(\tilde{\mu}(u))|\leq 2$ for all $u \in V(P_n+K_{i-1,j})$ and $\bigcup_{a \in V(\mu(u))}A_a$ contains a red vertex when $u\in B$; a blue vertex when $u \in C$; and a red and a blue vertex when $u \in V(P_n)$.
	Let~$G^{(i,j)}$ be obtained by taking $5\ell$ copies of~$G^{(i-1,j)}$ plus a red vertex~$v$ that is adjacent to all the blue vertices.
	Then $G^{(i,j)}$ has radius $2$ and pathwidth at most $k+1$.
	As in the base case, let $(A_x : x \in V(H))$ be a layered~$H$-partition of~$G^{(i,j)}$ of width~$\ell$.
	Let $w \in V(H)$ be such that $v \in A_w$.
	Then $A_w$ contains a red vertex.
	Since $G^{(i,j)}$ has radius $2$, $|A_w-\{v\}|\leq 5\ell-1$.
	Thus, there is a copy of $G^{(i-1,j)}$ that contains no vertices from~$A_w$.
	Now consider the sub-quotient $\tilde{H}$ of $H$ with respect to this copy of $G^{(i-1,j)}$.
	By induction, $\tilde{H}$ contains a model $\tilde{\mu}$ which satisfies the induction hypothesis.
	Let $\mu(b_i)=\{w\}$ and $\mu(v)=\tilde{\mu}(v)$ for all $v \in V(P_n+K_{i-1,j})$.
	Since~$v$ is adjacent to all the blue vertices in $G$, $w$ is adjacent to a vertex in $\bigcup_{a\in V(\mu(u))}A_a$ whenever $u \in V(P_n)\cup C$. Thus $\mu$ is a model of $P_n+K_{i,j}$ in $H$ which satisfies the induction hypothesis, as required.
\end{proof}

\begin{figure}[!h]
    \centering 
    \includegraphics[width=0.85\textwidth]{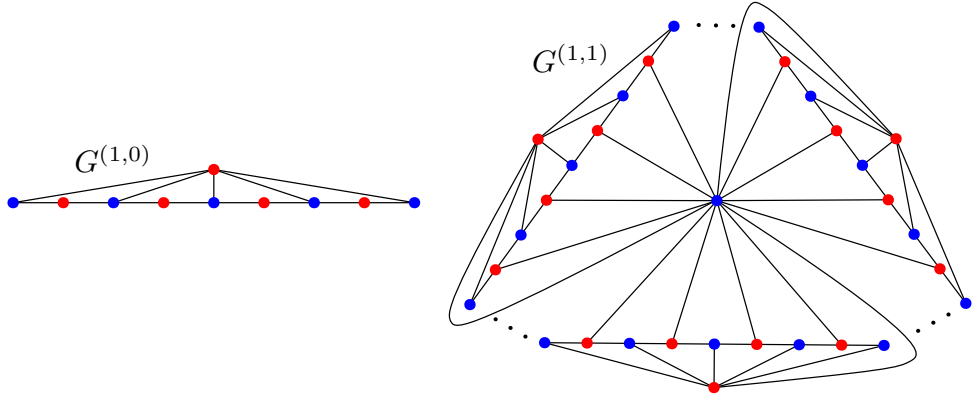}
    \caption{The graphs $G^{(1,0)}$ and $G^{(1,1)}$ from \cref{BipartiteLower}.}
    \label{fig:BipartiteLower}
\end{figure}

We now highlight several consequences of \cref{BipartiteLower}. First, when $i=1$ and $j=0$, the graph $G^{(1,0)}$ is an outerplanar squaregraph as illustrated in \cref{fig:BipartiteLower}. Since $P_2+K_{1,0}$ is a $3$-cycle, we have the following:

\begin{cor}
For every $\ell \in \N$, there exists a squaregraph $G$ such that for any graph $H$ and path $P$, if $G \subsetsim H \boxtimes P \boxtimes K_{\ell}$ then $H$ contains a cycle of length at most 6.
\end{cor}

Thus \Cref{squaregraphs} is best possible in the sense that ``outerplanar graph'' cannot be replaced by ``forest''.
 
Second, when $i=j=1$, the graph $G^{(1,1)}$ is a bipartite planar graph, as illustrated in 
\cref{fig:BipartiteLower}. Since $P_2+K_{1,1}\cong K_4$ which has treewidth $3$, we have the following:

\begin{cor}
    For every $\ell \in \N$, there exists a bipartite planar graph $G$ such that 
    for any graph $H$ and path $P$, if $G \subsetsim H \boxtimes P \boxtimes K_{\ell}$ then $H$ contains a 2-small minor of $K_4$ and 
thus $\tw(H)\geq 3$.
\end{cor}


Therefore, the maximum row treewidth of bipartite planar graphs is at least $3$. We conclude this section with the following open problem: what is the maximum row treewidth of bipartite planar graphs? As in the case of (non-bipartite) planar graphs, the answer is in $\{3,4,5,6\}$.

\section{Infinite Squaregraphs}
\label{Infinite}

In this section by `graph' we mean a graph $G$ with $V(G)$ finite or countably infinite. \citet{HMSTW} showed how \cref{PGPST} can be used to construct a graph that contains every planar graph as a subgraph and has several interesting properties. Here we adapt their methods to construct an analogous graph that contains every squaregraph as a subgraph. 

\citet{BCE10} gave several equivalent definitions of an infinite squaregraph. The following definition suits our purposes. Let $G$ be a locally finite\footnote{A graph $G$ is \defn{locally finite} if every vertex of $G$ has finite degree.} graph. For every vertex $v$ of $G$ and every $r\in\mathbb{N}$ the subgraph $G[ \{ w\in V(G): \dist_G(v,w)\leq r\} ]$ is called a \defn{ball}. Since $G$ is locally finite, every ball is finite. An infinite graph $G$ is a \defn{squaregraph} if it is locally finite and every ball in $G$ is a squaregraph. Let $\overrightarrow{P}$ be the 1-way infinite path, which has vertex-set $\mathbb{N}_0$ and edge-set $\{ \{i,i+1\} : i \in\mathbb{N}_0 \}$. It is well known that there is a \defn{universal} outerplanar graph $O$. This means that $O$ is outerplanar and every outerplanar graph is isomorphic to a subgraph of $O$. See Theorem~4.14 in \citep{HMSTW} for an explicit definition of $O$. 

\begin{thm}
\label{InfiniteSquaregraph}
Every squaregraph is isomorphic to a subgraph of $O\Bow \overrightarrow{P}$. 
\end{thm}

\cref{InfiniteSquaregraph} follows from \cref{squaregraphs} and the next lemma, which is an adaptation of Lemma~5.3 in \citep{HMSTW}.

\begin{lem}
Let $H$ be a graph. Let $G$ be a locally finite graph such that $B\subsetsim H\Bow \overrightarrow{P}$ for every ball $B$ in $G$. Then $G\subsetsim H \Bow \overrightarrow{P}$.
\end{lem}

\begin{proof}[Proof Sketch] 
Fix $v\in V(G)$. For $n\in\mathbb{N}_0$, let $V_n :=\{w\in V(G): \dist_G(v,w)=n\}$ and $G_n:=G[ V_0\cup V_1 \cup\dots \cup V_n]$. So $G_n$ is a finite ball in $G$. By assumption, $G_n \subsetsim H \Bow \overrightarrow{P}$. Let $X_n$ be the set of all thin layered $H$-partitions $(\Pcal,\mathcal{L})$ of $G_n$, such that $L$ is an independent set in $G_n$ for each $L \in \mathcal{L}$. By \cref{BowPartitions}, $X_n\neq\emptyset$. Since $G_n$ is finite and connected, $X_n$ is finite. 
For each $n\in\mathbb{N}$ and for each $(\Pcal,\mathcal{L})\in X_n$, 
if 
$\Pcal':= \{ Y \setminus V_n : Y \in \Pcal, Y \setminus V_n\neq\emptyset\}$
and
$\Lcal':= \{ L \setminus V_n : L \in \Lcal, Y \setminus V_n\neq\emptyset\}$
then 
$(\Pcal',\Lcal')\in X_{n-1}$ (since $G_{n-1}$ is connected). 
By K\H{o}nig's Lemma, there is an infinite sequence $(\Pcal_0,\Lcal_0), 
(\Pcal_1,\Lcal_1), (\Pcal_2,\Lcal_2), \dots$ where 
$\Pcal_{n-1}=\Pcal'_{n}$
and
$\Lcal_{n-1}=\Lcal'_{n}$ for each $n\in\mathbb{N}$. 
By construction, 
$\Pcal_{n-1}$ is a `sub-partition' of $\Pcal_{n}$ and
$\Lcal_{n-1}$ is a `sub-partition' of $\Lcal_{n}$. 
Let $\Pcal:= \bigcup_{n\in\mathbb{N}_0} \Pcal_n$ and
 $\Lcal:= \bigcup_{n\in\mathbb{N}_0} \Lcal_n$. Then $(\Pcal,\Lcal)$ is a thin layered $H$-partition of $G$; see \citep{HMSTW} for details. By \cref{BowPartitions}, $G\subsetsim H \Bow \overrightarrow{P}$.
\end{proof}

\subsection*{Acknowledgement}

This research was initiated at the workshop, \emph{Geometric Graphs and Hypergraphs}, 30 August -- 3 September 2021, organised by Torsten Ueckerdt and Lena Yuditsky. Thanks to the organisers and other participants for creating a productive environment.

\fontsize{11}{12}
\selectfont
\let\oldthebibliography=\thebibliography
\let\endoldthebibliography=\endthebibliography
\renewenvironment{thebibliography}[1]{%
	\begin{oldthebibliography}{#1}%
		\setlength{\parskip}{0.3ex}%
		\setlength{\itemsep}{0.3ex}%
	}{\end{oldthebibliography}}
\bibliographystyle{DavidNatbibStyle}
\bibliography{myBibliography}

\def\soft#1{\leavevmode\setbox0=\hbox{h}\dimen7=\ht0\advance \dimen7
  by-1ex\relax\if t#1\relax\rlap{\raise.6\dimen7
  \hbox{\kern.3ex\char'47}}#1\relax\else\if T#1\relax
  \rlap{\raise.5\dimen7\hbox{\kern1.3ex\char'47}}#1\relax \else\if
  d#1\relax\rlap{\raise.5\dimen7\hbox{\kern.9ex \char'47}}#1\relax\else\if
  D#1\relax\rlap{\raise.5\dimen7 \hbox{\kern1.4ex\char'47}}#1\relax\else\if
  l#1\relax \rlap{\raise.5\dimen7\hbox{\kern.4ex\char'47}}#1\relax \else\if
  L#1\relax\rlap{\raise.5\dimen7\hbox{\kern.7ex
  \char'47}}#1\relax\else\message{accent \string\soft \space #1 not
  defined!}#1\relax\fi\fi\fi\fi\fi\fi}
\begin{thebibliography}{29}
\providecommand{\natexlab}[1]{#1}
\providecommand{\msn}[1]{MR:\,\href{http://www.ams.org/mathscinet-getitem?mr=MR{#1}}{#1}}
\providecommand{\ZBL}[1]{Zbl:\,\href{https://www.zentralblatt-math.org/zmath/en/search/?q=an:#1}{#1}}
\providecommand{\url}[1]{\texttt{#1}}
\providecommand{\urlprefix}{}
\expandafter\ifx\csname urlstyle\endcsname\relax
  \providecommand{\doi}[1]{doi:\discretionary{}{}{}#1}\else
  \providecommand{\doi}{doi:\discretionary{}{}{}\begingroup
  \urlstyle{rm}\Url}\fi

\bibitem[{Alon et~al.(2002)Alon, Grytczuk, Ha{\l}uszczak, and Riordan}]{AGHR02}
\textsc{Noga Alon, {Jaros{\l}aw} Grytczuk, Mariusz Ha{\l}uszczak, and Oliver
  Riordan}.
\newblock \href{https://doi.org/10.1002/rsa.10057}{Nonrepetitive colorings of
  graphs}.
\newblock \emph{Random Structures Algorithms}, 21(3--4):336--346, 2002.

\bibitem[{Bandelt and Chepoi(2008)}]{BC08}
\textsc{Hans-J\"{u}rgen Bandelt and Victor Chepoi}.
\newblock \href{https://doi.org/10.1090/conm/453/08795}{Metric graph theory and
  geometry: a survey}.
\newblock In \emph{Surveys on discrete and computational geometry}, vol. 453 of
  \emph{Contemp. Math.}, pp. 49--86. Amer. Math. Soc., 2008.

\bibitem[{Bandelt et~al.(2010)Bandelt, Chepoi, and Eppstein}]{BCE10}
\textsc{Hans-J\"{u}rgen Bandelt, Victor Chepoi, and David Eppstein}.
\newblock \href{https://doi.org/10.1137/090760301}{Combinatorics and geometry
  of finite and infinite squaregraphs}.
\newblock \emph{SIAM J. Discrete Math.}, 24(4):1399--1440, 2010.

\bibitem[{Bannister et~al.(2019)Bannister, Devanny, Dujmovi\'c, Eppstein, and
  Wood}]{BDDEW19}
\textsc{Michael~J. Bannister, William~E. Devanny, Vida Dujmovi\'c, David
  Eppstein, and David~R. Wood}.
\newblock \href{https://doi.org/10.1007/s00453-018-0487-5}{Track layouts,
  layered path decompositions, and leveled planarity}.
\newblock \emph{Algorithmica}, 81(4):1561--1583, 2019.

\bibitem[{Bonamy et~al.(2020)Bonamy, Gavoille, and Pilipczuk}]{BGP20}
\textsc{Marthe Bonamy, Cyril Gavoille, and Michał Pilipczuk}.
\newblock \href{https://doi.org/10.1137/1.9781611975994.27}{Shorter labeling
  schemes for planar graphs}.
\newblock In \textsc{Shuchi Chawla}, ed., \emph{Proc. ACM-SIAM Symp. on
  Discrete Algorithms (SODA '20)}, pp. 446--462. 2020.

\bibitem[{Bonnet et~al.(2022)Bonnet, Kwon, and Wood}]{BKW}
\textsc{\'Edouard Bonnet, {O-joung} Kwon, and David~R. Wood}.
\newblock \href{http://arxiv.org/abs/2202.11858}{Reduced bandwidth: a
  qualitative strengthening of twin-width in minor-closed classes (and
  beyond)}.
\newblock 2022, arXiv:2202.11858.

\bibitem[{Bose et~al.(2020)Bose, Dujmovi\'c, Javarsineh, and Morin}]{BDJM}
\textsc{Prosenjit Bose, Vida Dujmovi\'c, Mehrnoosh Javarsineh, and Pat Morin}.
\newblock \href{http://arxiv.org/abs/2007.06455}{Asymptotically optimal vertex
  ranking of planar graphs}.
\newblock arXiv:2007.06455, 2020.

\bibitem[{Bose et~al.(2021)Bose, Dujmovi\'c, Javarsineh, Morin, and
  Wood}]{BDJMW}
\textsc{Prosenjit Bose, Vida Dujmovi\'c, Mehrnoosh Javarsineh, Pat Morin, and
  David~R. Wood}.
\newblock \href{http://arxiv.org/abs/2105.01230}{Separating layered treewidth
  and row treewidth}.
\newblock arXiv:2105.01230, 2021.

\bibitem[{Diestel(2018)}]{Diestel5}
\textsc{Reinhard Diestel}.
\newblock \href{http://diestel-graph-theory.com/}{Graph theory}, vol. 173 of
  \emph{Graduate Texts in Mathematics}.
\newblock Springer, 5th edn., 2018.

\bibitem[{Distel et~al.(2021)Distel, Hickingbotham, Huynh, and Wood}]{DHHW}
\textsc{Marc Distel, Robert Hickingbotham, Tony Huynh, and David~R. Wood}.
\newblock \href{http://arxiv.org/abs/2112.10025}{Improved product structure for
  graphs on surfaces}.
\newblock 2021, arXiv:2112.10025.

\bibitem[{D\k{e}bski et~al.(2021)D\k{e}bski, Felsner, Micek, and
  Schr\"{o}der}]{DFMS21}
\textsc{Micha{\l} D\k{e}bski, Stefan Felsner, Piotr Micek, and Felix
  Schr\"{o}der}.
\newblock \href{https://doi.org/10.19086/aic.27351}{Improved bounds for
  centered colorings}.
\newblock \emph{Adv. Comb.}, \#8, 2021.

\bibitem[{Dujmovi\'c et~al.(2021)Dujmovi\'c, Esperet, Gavoille, Joret, Micek,
  and Morin}]{DEJGMM21}
\textsc{Vida Dujmovi\'c, Louis Esperet, Cyril Gavoille, Gwena\"el Joret, Piotr
  Micek, and Pat Morin}.
\newblock \href{https://doi.org/10.1145/3477542}{Adjacency labelling for planar
  graphs (and beyond)}.
\newblock \emph{J. ACM}, 68(6):42, 2021.

\bibitem[{Dujmovi{\'c} et~al.(2020{\natexlab{a}})Dujmovi{\'c}, Esperet, Joret,
  Walczak, and Wood}]{DEJWW20}
\textsc{Vida Dujmovi{\'c}, Louis Esperet, Gwena\"{e}l Joret, Bartosz Walczak,
  and David~R. Wood}.
\newblock \href{https://doi.org/10.19086/aic.12100}{Planar graphs have bounded
  nonrepetitive chromatic number}.
\newblock \emph{Adv. Comb.}, \#5, 2020{\natexlab{a}}.

\bibitem[{Dujmovi{\'c} et~al.(2022)Dujmovi{\'c}, Esperet, Morin, Walczak, and
  Wood}]{DEMWW22}
\textsc{Vida Dujmovi{\'c}, Louis Esperet, Pat Morin, Bartosz Walczak, and
  David~R. Wood}.
\newblock \href{https://doi.org/10.1017/S0963548321000213}{Clustered
  3-colouring graphs of bounded degree}.
\newblock \emph{Combin. Probab. Comput.}, 31(1):123--135, 2022.

\bibitem[{Dujmovi{\'c} et~al.(2020{\natexlab{b}})Dujmovi{\'c}, Joret, Micek,
  Morin, Ueckerdt, and Wood}]{DJMMUW20}
\textsc{Vida Dujmovi{\'c}, Gwena\"{e}l Joret, Piotr Micek, Pat Morin, Torsten
  Ueckerdt, and David~R. Wood}.
\newblock \href{https://doi.org/10.1145/3385731}{Planar graphs have bounded
  queue-number}.
\newblock \emph{J. ACM}, 67(4):\#22, 2020{\natexlab{b}}.

\bibitem[{Dujmovi\'c et~al.(2022)Dujmovi\'c, Morin, and Odak}]{DMO}
\textsc{Vida Dujmovi\'c, Pat Morin, and Saeed Odak}.
\newblock \href{http://arxiv.org/abs/2202.12882}{Odd colourings of graph
  products}.
\newblock 2022, arXiv:2202.12882.

\bibitem[{Dujmovi{\'c} et~al.(2019)Dujmovi{\'c}, Morin, and Wood}]{DMW}
\textsc{Vida Dujmovi{\'c}, Pat Morin, and David~R. Wood}.
\newblock \href{http://arxiv.org/abs/1907.05168}{Graph product structure for
  non-minor-closed classes}.
\newblock 2019, arXiv:1907.05168.

\bibitem[{Esperet et~al.(2020)Esperet, Joret, and Morin}]{EJM}
\textsc{Louis Esperet, Gwena\"{e}l Joret, and Pat Morin}.
\newblock \href{http://arxiv.org/abs/2010.05779}{Sparse universal graphs for
  planarity}.
\newblock 2020, arXiv:2010.05779.

\bibitem[{Garman et~al.(1976)Garman, Ringeisen, and White}]{GRW76}
\textsc{B.~L. Garman, R.~D. Ringeisen, and A.~T. White}.
\newblock \href{https://doi.org/10.4153/CJM-1976-052-9}{On the genus of strong
  tensor products of graphs}.
\newblock \emph{Canadian J. Math.}, 28(3):523--532, 1976.

\bibitem[{Harvey and Wood(2017)}]{HW17}
\textsc{Daniel~J. Harvey and David~R. Wood}.
\newblock \href{https://doi.org/10.1002/jgt.22030}{Parameters tied to
  treewidth}.
\newblock \emph{J. Graph Theory}, 84(4):364--385, 2017.

\bibitem[{Hickingbotham and Wood(2021)}]{HW21b}
\textsc{Robert Hickingbotham and David~R. Wood}.
\newblock \href{http://arxiv.org/abs/2111.12412}{Shallow minors, graph products
  and beyond planar graphs}.
\newblock 2021, arXiv:2111.12412.

\bibitem[{Hong et~al.(2021)Hong, Lai, and Liu}]{HLL21}
\textsc{Zhen-Mu Hong, Hong-Jian Lai, and Jianbing Liu}.
\newblock \href{https://doi.org/10.1002/jgt.22692}{Induced subgraphs of product
  graphs and a generalization of {H}uang's theorem}.
\newblock \emph{J. Graph Theory}, 98(2):285--308, 2021.

\bibitem[{Huynh et~al.(2021)Huynh, Mohar, {\v{S}}{\'a}mal, Thomassen, and
  Wood}]{HMSTW}
\textsc{Tony Huynh, Bojan Mohar, Robert {\v{S}}{\'a}mal, Carsten Thomassen, and
  David~R. Wood}.
\newblock \href{http://arxiv.org/abs/2109.00327}{Universality in minor-closed
  graph classes}.
\newblock 2021, arXiv:2109.00327.

\bibitem[{K{\"u}ndgen and Pelsmajer(2008)}]{KP-DM08}
\textsc{Andr\'e K{\"u}ndgen and Michael~J. Pelsmajer}.
\newblock \href{https://doi.org/10.1016/j.disc.2007.08.043}{Nonrepetitive
  colorings of graphs of bounded tree-width}.
\newblock \emph{Discrete Math.}, 308(19):4473--4478, 2008.

\bibitem[{Reed(1997)}]{Reed97}
\textsc{Bruce~A. Reed}.
\newblock \href{https://doi.org/10.1017/CBO9780511662119.006}{Tree width and
  tangles: a new connectivity measure and some applications}.
\newblock In \textsc{R.~A. Bailey}, ed., \emph{Surveys in combinatorics}, vol.
  241 of \emph{London Math. Soc. Lecture Note Ser.}, pp. 87--162. Cambridge
  Univ. Press, 1997.

\bibitem[{Soltan et~al.(1973)Soltan, Zambitski\u{\i}, and Prisakaru}]{SZP73}
\textsc{P.~S. Soltan, D.~K. Zambitski\u{\i}, and K.~F. Prisakaru}.
\newblock \textcyr{Ekstremal\cprime nye zadachi na grafakh i algoritmy ikh
  resheniya} [{E}xtremal problems on graphs and algorithms of their solution].
\newblock Izdat. ``\v{S}tiinca'', Kishinev, 1973.

\bibitem[{Sugiyama et~al.(1981)Sugiyama, Tagawa, and Toda}]{STT81}
\textsc{K.~Sugiyama, S.~Tagawa, and M.~Toda}.
\newblock Methods for visual understanding of hierarchical system structures.
\newblock \emph{IEEE Trans. Systems Man Cybernet.}, 11(2):109--125, 1981.

\bibitem[{Ueckerdt et~al.(2021)Ueckerdt, Wood, and Yi}]{UWY}
\textsc{Torsten Ueckerdt, David~R. Wood, and Wendy Yi}.
\newblock \href{http://arxiv.org/abs/2108.00198}{An improved planar graph
  product structure theorem}.
\newblock 2021, arXiv:2108.00198.

\bibitem[{Wood(2021)}]{Wood21}
\textsc{David~R. Wood}.
\newblock \href{https://doi.org/10.37236/9777}{Nonrepetitive graph colouring}.
\newblock \emph{Electron. J. Combin.}, DS24, 2021.

\end{thebibliography}
\end{document}